\documentclass[12pt,twoside]{amsart}
\usepackage{amsmath,amssymb,amsfonts,amsthm,amsopn}
\usepackage{graphics}

\newcommand{\modsp}{modulation space}

\def\epsilon{\varepsilon}
\def\a{\alpha}

\newtheorem{theorem}{Theorem}[section]
\newtheorem{lemma}[theorem]{Lemma}
\newtheorem{corollary}[theorem]{Corollary}
\newtheorem{proposition}[theorem]{Proposition}
\newtheorem{definition}[theorem]{Definition}

\newcommand{\beqa}{\begin{eqnarray*}}
\newcommand{\eeqa}{\end{eqnarray*}}

\newcommand{\field}[1]{\mathbb{#1}}
\newcommand{\bR}{\field{R}}        %  real numbers
\newcommand{\bN}{\field{N}}        %  natural numbers
\def\la{\lambda}

\def\eps{\epsilon}

\def\cS{\mathcal{S}}

\def\cM{\mathcal{M}}

\def\cC{\mathcal{C}}

\def\Fu{\mathfrak{F}}

\def\rd{\bR^d}

\def\rdd{{\bR^{2d}}}

\def\intrd{\int_{\rd}}

\def\R{\right)}

\def\<{\left<}
\def\>{\right>}

\def\mv1{M_v^1}

\def\mn{(m,n)}
\def\mn'{(m',n')}

\def\o{\eta}

\def\R{\mathbb{R}}
\def\Ren{\mathbb{R}^d}
\def\Renn{\mathbb{R}^{2d}}

\def\Fur{\mathcal{F}}

\def\Sn2{S_{2}(L^{2}(\Ren))}
\def\S1{S_{1}(L^{2}(\Ren))}
\def\sig00{\sigma_{0,0}}
\def\la{\langle}
\def\ra{\rangle}

\begin{document}
\begin{abstract}
We consider a class of linear Schr\"odinger equations in $\rd$, with analytic symbols. We prove a global-in-time integral representation for the corresponding propagator as a generalized Gabor multiplier with a window analytic and decaying exponentially at infinity, which is transported by the Hamiltonian flow. We then provide three applications of the above result: the exponential sparsity  in phase space of the corresponding propagator with respect to Gabor wave packets, a wave packet characterization of Fourier integral operators with analytic phases and symbols, and the propagation of analytic singularities.     
\end{abstract}

\title[Wave packet analysis of Schr\"odinger equations]{Wave packet analysis of Schr\"odinger equations in analytic function spaces}

\author{Elena Cordero,  Fabio Nicola  and Luigi Rodino}
\address{Dipartimento di Matematica,
Universit\`a di Torino, via Carlo Alberto 10, 10123 Torino, Italy}
\address{Dipartimento di Scienze Matematiche,
Politecnico di Torino, corso Duca degli Abruzzi 24, 10129 Torino,
Italy}
\address{Dipartimento di Matematica,
Universit\`a di Torino, via Carlo Alberto 10, 10123 Torino, Italy}

\email{elena.cordero@unito.it}
\email{fabio.nicola@polito.it}
\email{luigi.rodino@unito.it}


%\date{}
\maketitle

\section{Introduction}
Consider the Cauchy problem 
\begin{equation}\label{equazione}
\begin{cases}
D_t u+ a^w(t,x,D)u=0\\
u(0)=u_0,
\end{cases}
\end{equation}
where $D_t=-i\partial_t$ and the real-valued symbol $a(t,x,\xi)$ is continuous in $t$ in some interval $[0,T]$ and smooth with respect to $x,\xi\in\rd$, satisfying
\begin{equation}\label{wtat}
|\partial^\alpha_z a(t,z)|\leq C_\alpha,\quad |\alpha|\geq 2,\ z\in \rdd,\ t\in[0,T]
\end{equation}
(Weyl quantization is understood). As a typical model one can consider the case when $a(t,x,\xi)$ is a quadratic form in $x,\xi$, which gives rise to metaplectic operators. Equations of this type turned out to be important in spectral theory \cite{helffer84,helffer-rob1} and in regularization issues for equations with rough coefficients \cite{tatarustr}. Depending on the applications, several additional conditions are imposed on the symbol $a$ and its derivatives. In any case a fundamental problem is to obtain some integral representation of the propagator, from which one can then deduce estimates for the solutions. In principle, {\it for small time} one expects the propagator to be represented by a Fourier integral operator (FIO) 
\begin{equation}\label{classicalFIO}
S(t,0)f=(2\pi)^{-d}\int_{\rd} e^{i\Phi(t,x,\eta)}\sigma(t,x,\eta)\, \widehat{f}{(\eta)}\,d\eta
\end{equation}
with a smooth  real-valued phase $\Phi(t,x,\xi)$, having quadratic growth with respect to the variables $x,\xi$, and a symbol $\sigma(t,\cdot)\in S^0_{0,0}$, i.e.\ bounded together with its derivatives. This was first proved in \cite{chazarain,helffer84,helffer-rob1} for a class of symbols $a(t,x,\xi)$ of polyhomogeneous type, i.e.\ with an asymptotic expansion in homogeneous terms in $z=(x,\xi)$, of decreasing order. These decay conditions are essential for the symbolic calculus to work and the integral representation was in fact constructed by the WKB method. Recently in \cite{cgnr} the above representation \eqref{classicalFIO} was proved to be true for small time only under the assumption \eqref{wtat}. Such a representation however does no longer keep valid for large time because of the appearance of caustics; in other terms, the space of FIOs is not an algebra. Different approaches have been proposed by several authors, see e.g.\ \cite{asada-fuji, bony1,bony2,bony3,cgnr2,fu,graffi,kki1,kt,marzuola,parmeggiani,tataru,treves}. \par
In the case of analytic symbols, which is the framework of this paper, there are further technical difficulties and surprisingly, to our knowledge, it is not even known whether the {\it exact} integral representation \eqref{classicalFIO} holds for an analytic phase and symbol, at least for small time. We will answer positively this question as a byproduct of more general results.\par Namely, consider the analytic symbol classes $S^{(k)}_{a}$, $k\in\bN$, defined by the estimates
\begin{equation}\label{s200}
|\partial^\alpha_{\xi}\partial^\beta_x a(x,\xi)|\leq C^{|\alpha|+|\beta|+1}\alpha!\beta!,\quad |\alpha|+|\beta|\geq k,\quad x,\xi\in\rd,
\end{equation}
endowed with the obvious inductive limit topology of Fr\'echet spaces.\par
 Let now $T>0$ be fixed and consider therefore a symbol $a(t,x,\xi)$, $t\in[0,T]$, $x,\xi\in\rd$, satisfying the following conditions:
\begin{itemize}
\item[\bf (i)] $a(t,x,\xi)$ is real-valued, $t\in[0,T]$, $x,\xi\in\rd$;
\item[\bf (ii)] $a(t,\cdot)$ belongs to a bounded subset of $S^{(2)}_{a}$ for $t\in[0,T]$;
\item[\bf (iii)] the map $t\mapsto a(t,\cdot)$ is (weakly) continuous from $[0,T]$ to $S'(\rd)$ (or equivalently pointwise).
\end{itemize}
As a very simple example, one may consider the operator $a(t,x,D)=-\Delta+V(t,x)$,
where the potential $V(t,x)$ is real-valued, continuous with respect to $t$, and verifying $|\partial^\alpha_x V(t,x)|\leq C^{|\alpha|+1}\alpha!$ for $|\alpha|\geq 2$ (cf.\ \cite{kki4}). \par
Under the above hypothesis it is easy to show by the usual energy method that the Cauchy problem \eqref{equazione} is wellposed in $\cS(\rd)$ (cf.\ \cite{tataru}). More generally, one can consider the strongly continuous propagator
\[
S(t,s):\cS(\rd)\to\cS(\rd),\quad 0\leq s\leq t\leq T,
\]
which maps the initial datum at time $s$ to the solution at time $t$. 
\par
In order to state our main result, let us fix some notation.  \par
For $x,\xi\in\rd$ we define the phase space shifts
\[
\pi(x,\xi)f= T_x M_\xi f,
\]
where $T_x f(y)=f(y-x)$ and $M_\xi f(y)=e^{i\xi y} f(y)$ are the translation and modulation operators.\par
Moreover we consider the Hamiltonian flow $(x^t,\xi^t)$, as a function of $t\in[0,T]$, $x,\xi\in\rd$, given by the solution of
\begin{equation}\label{sistema}
\begin{cases}
\dot{x}^t=a_\xi(t,x^t,\xi^t)\\
\dot{\xi}^t=-a_x(t,x^t,\xi^t) \\
x^0(x,\xi)=x,\ \xi^0(x,\xi)=\xi.
\end{cases}
\end{equation}
Further consider the real-valued phase $\psi(t,x,\xi)$ defined by
\begin{equation}\label{fase}
\psi(t,x,\xi)= \int_0^t \Big(\xi^s a_\xi (s,x^s,\xi^s)-a(s,x^s,\xi^s)\Big)\,ds.
\end{equation}
We also define the Gelfand-Shilov space \cite{GS}
\begin{equation}\label{gs}
S^1_1(\rd)=\{f\in\cS(\rd):\ |x^\alpha\partial^\beta f(x)|\leq C^{|\alpha|+|\beta|+1}\alpha!\beta!\  \forall\alpha,\beta\in\bN^d,\ \textrm{for some}\ C>0\},
\end{equation}
with the inductive limit topology. Functions in $S^1_1(\rd)$ are analytic and decay exponentially at infinity, and the same holds for their Fourier transform.
 \par
The following result gives a global-in-time representation of the corresponding propagator.
\begin{theorem}\label{mainteo}
Fix any window $g\in S^1_1(\rd)$. Under the above assumptions ${\bf (i)-(iii)}$, the propagator $S(t,s)$ has the following integral representation: for $f\in\cS(\rd)$,
\begin{equation}\label{rap}
S(t,s)f=\int_{\rdd} e^{i\psi(t,x,\xi)-i\psi(s,x,\xi)}\pi(x^t,\xi^t) G(t,s,x,\xi,\cdot)\langle f,\pi(x^s,\xi^s)g\rangle\,dx\,d\xi,
\end{equation}
for $0\leq s\leq t\leq T$, for some window $G(t,s,x,\xi,y)$ such that each derivative $\partial^\alpha_x\partial^\beta_\xi G(t,s,x,\xi,y)$, $\alpha,\beta\in\bN^d$, belongs to a bounded subset of $S^1_1(\rd)$ as a function of $y$, when $0\leq s\leq t\leq T$, $x,\xi\in\rd$.  \par
%.
\end{theorem}
A similar representation in the smooth category was obtained by Tataru \cite{tataru} (when the window $g$ is Gaussian, but his argument extends to any $g\in\cS(\rd)$), see\ also \cite{kt,marzuola}. To be precise, if $g\in\cS(\rd)$ and ${\bf (ii)}$ is replaced by the weaker condition \eqref{wtat}, then the integral representation \eqref{rap} holds true with a window $G(t,s,x,\xi,y)$ which is Schwartz with respect to $y$ uniformly with respect to $t,s,x,\xi$ (moreover $G(t,s,x,\xi,\cdot)$ is continuous in $\cS(\rd)$ as a function of $s,t\in[0,T]$, for fixed $x,\xi\in\rd$; this was not stated explicitly but it follows easily from the proof). \par
Formally, our result therefore amounts to replacing the Schwartz space $\cS(\rd)$ by $S^1_1(\rd)$. However, to this end we will need much more refined energy estimates in certain analytic function spaces, which will be proved in Section 2. Incidentally, these estimates seem of particular interest in their own right and show that the radius of analyticity of the solution decreases at most exponentially. We plan to carry on this issue elsewhere, in the more general context of nonlinear Schr\"odinger equations. Instead here we present three applications of the above result which represent, in fact, our main motivation: \begin{itemize}
\item[\bf (a)] the  exponential sparsity of the corresponding Gabor matrix; 
\item[\bf (b)] the representation \eqref{classicalFIO} as a classical FIO away from caustics;
\item[\bf (c)] the propagation of analytic singularities.
\end{itemize}
 We now briefly discuss these applications. \par\smallskip
{\bf (a)} Almost diagonalization of pseudodifferential operators via Gabor wave packets \cite{g-ibero,GR,rochberg,tataru} represents an important contribution of Time-frequency analysis to PDEs. The case of FIO of the above type was considered in \cite{cgnr,cnr1,tataru} and the Gabor matrix of such an operator was proved to be highly concentrated along the graph of the corresponding canonical transformation. In particular, for smooth symbols and phases one obtains super-polynomial decay. Recently we considered the problem of the {\it exponential} sparsity for a large class of constant coefficient evolution operators \cite{cnr2} and of classical FIOs with analytic phases and symbols \cite{cnribero}. Now, it follows from the above representation in Theorem \ref{mainteo} that similarly the propagator $S(t,s)$ displays an exponential sparsity. Namely, if $g\in S^1_1(\rd)$, under the above assumptions we have the estimate
\begin{equation}\label{sparsityeq0}
|\langle S(t,0)\pi(z)g,\pi(w)g\rangle|\leq C\exp\big(-\epsilon|w-\chi_t(z)|\big),\quad w,z\in\rdd,\ 0\leq t\leq T,
\end{equation}
for some constants $C,\epsilon>0$, where $\chi_t(x,\xi)=(x^t,\xi^t)$ is the corresponding canonical transformation. \par
An immediate consequence is the continuity of $S(t,0)$ on a large class of weighted modulation spaces \cite{book,baoxiang}, whose weight may grow even exponentially (which corresponds to analyticity or exponential decay for the functions in those spaces). We refer to Section 4 below for their definition and the precise statement (Corollary 4.3). Here we only mention the papers \cite{bertinoro3,cgnr,cgnr2,cn,cnr1,kki4,MNRTT,ruz,bertinoro57,bertinoro58,bertinoro58bis,baoxiang} devoted to the continuity of the propagator $e^{it\Delta}$ and generalizations on weighted modulation spaces, in the case of weights with polynomial growth. Incidentally we observe that modulations spaces with exponential weighs were also used with success in \cite{bertinoro57,baoxiang} to quantify the smoothing effect, of infinite order, of the heat semigroup $e^{t\Delta}$ for $t>0$ (cf.\ also \cite{lerner2,morimoto}). \par\medskip
{\bf (b)} As anticipated we can come back to the classical Fourier representation \eqref{classicalFIO} at least for those values of $t$ such that ${\rm det}\, \partial x^t/\partial x\not=0$. In fact, we can prove the following remarkable {\it characterization}. \par\medskip
{\it At every instant time when ${\rm det}\, \partial x^t/\partial x\not=0$, the estimate \eqref{sparsityeq0} turns out to be equivalent to the integral representation \eqref{classicalFIO} for a phase $\Phi(t,\cdot)$ corresponding to $\chi_t$ and some symbol $\sigma(t,\cdot)\in S^{(0)}_a$.}  
\par\medskip
Notice that even for the special case of nice symbols, e.g.\ with a polyhomogeneous expansion, we could not deduce this claim via a WKB construction, because a global symbolic calculus in $\rd$ is not available in the analytic category. \par\medskip
{\bf (c)} The sparsity estimate \eqref{sparsityeq0} is definitively a result of propagation of analytic singularities. This will be made explicit in terms of the filter of the singularities $\mathfrak{F}(f)$, $f\in (S^1_1)'(\rd)$, i.e.\ the system of neighborhoods at infinity of the analytic spectrum of $f$ in $\rdd$, cf.\ \cite{rodino82} and Definitions \ref{def7.3}, \ref{def7.7} in the sequel. Namely we shall prove 
\begin{equation}\label{9.bis} 
\chi_t(\mathfrak{F}(f))= \mathfrak{F}(S(t,0)f).
\end{equation}
The result is optimal in the absence of further assumptions on $a(t,x,\xi)$ and $\chi_t$. When $a(t,x,\xi)$ has additional structure, one can rephrase \eqref{9.bis} by fixing a compactification of $\rdd$, having stability with respect to $\chi_t$. The simplest and perhaps most natural case is, concerning polyhomogeneous symbols, the compactifications by a sphere at infinity, see \cite{sato69} and in the global setting \cite{hormander}. We address to \cite{mizuhara} for a rich bibliography on this subject, and we leave to the reader to restate \eqref{9.bis} in the polyhomogeneous case.  
\par\medskip
Briefly, the paper is organized as follows. In Section 2 we prove the above mentioned infinite order energy estimate. In Section 3 we will discuss the integral representation in \eqref{rap} in the special case when the symbol $a(t,x,\xi)$ in \eqref{equazione} is a second order polynomial. It turns out that the window $G$ is then independent of $x,\xi$ and, in fact,  $G=S(t,s)g$. This also serves as an illustration of the more involved arguments in the subsequent section. Section 4 is devoted to the proof of Theorem \ref{mainteo}. In Section 5 we prove the above mentioned sparsity result of the Gabor matrix and will deduce, as a consequence, continuity on modulation spaces with exponential growth. In Section 6 we come back to the representation \eqref{classicalFIO} as classical FIO. Finally in Section 7 we study the problem of propagation of analytic singularities. \par\medskip
{\bf Notation}
The Fourier transform is normalized as $$\Fur(f)(\xi)=\widehat{f}(\xi)=\int_{\rd} e^{-ix\xi} f(x)\,dx.$$
\par
We will denote by $\langle\cdot,\cdot\rangle$ the inner product in $L^2(\rd)$ or the duality bracket (linear in the first factor).
We have already defined in the Introduction the phase space shifts $\pi(x,\xi)f=M_\xi T_x f$. Given $f,g$ in spaces in duality, e.g.\ $f\in\cS'(\rd)$, $g\in\cS(\rd)$, the short-time Fourier transform (STFT), or Bargman transform, of $f$ with window $g$ is defined as
\begin{equation}\label{stft}
V_g f(z)=\langle f,\pi(z)g\rangle,\quad z=(x,\xi)\in\rd\times\rd.
\end{equation}
The Weyl quantization of a symbol $a(x,\xi)$ is defined as
\[
a^w(x,D)f=(2\pi)^{-d}\iint e^{i(x-y)\xi}a\Big(\frac{x+y}{2},\xi\Big)f(y)\,dy\,d\xi.
\]
We recall that real-valued symbols give rise to formally self-adjoint operators.\par
The symbol classes $S^{(k)}_{a}$, $k\in\bN$, were already introduced above in \eqref{s200}.\par
\section{Infinite order energy estimates}
It is clear from the definition in \eqref{gs} that $f\in S^1_1(\rd)$ if and only if there exists $\epsilon>0$ such that the numerical sequence
\[
E^\epsilon_N[f]:=\sum_{|\alpha|+|\beta|= N}\frac{\epsilon^{|\alpha|+|\beta|}}{\alpha!\beta!}\|x^\alpha\partial^\beta f\|_{L^2},\quad N\in\mathbb{N},
\]
 is bounded (cf.\ \cite[Section 6.1]{nr}). We prove by induction on $N$ that this is the case for any solution to the problem \eqref{equazione}, if $u_0\in S^1_1(\rd)$. In the following result a certain uniformity of the constants involved is emphasized. This is essential in the applications below.
 \begin{theorem}\label{enerinf}
Let $a(t,x,\xi)$ satisfy conditions {\bf (i)} and ${\bf (iii)}$ in the Introduction and also the estimate\par\medskip
{\bf (ii)$^\prime$}  there exists a constant $C_1>0$ such that
\[
|\partial^\alpha_x\partial^\beta_\xi a(t,x,\xi)|\leq C_1^{|\alpha|+|\beta|+1}\alpha!\beta! (1+|x|+|\xi|),\ |\alpha|+|\beta|\geq 1,\ t\in[0,T],\ x,\xi\in\rd.
\]
Then there exist constants $\overline{\epsilon}_0,A>0$ depending only on the dimension $d$ and the constants $C_1$ in {\bf (ii)$^\prime$} such that for every solution $u\in C^1([0,T],\cS(\rd))$ to
\[
D_tu+a^w(t,x,D)u=0
\]
 we have
 \begin{equation}\label{stima00}
E^{\epsilon(t)}_N[u(t)]\leq 2 \sup_{0\leq k\leq N} E^{\epsilon_0}_k[u(0)],\quad N\in\bN,\ t\in[0,T],
\end{equation}
with
\begin{equation}\label{epst}
\epsilon(t)=\epsilon_0 e^{-At},
\end{equation} for every $0<\epsilon_0\leq \overline{\epsilon}_0$.
\end{theorem}
 \begin{proof}
Let $L=D_t+a^w(t,x,D)$. Since $a^w(t,x,D)$ is formally self-adjoint we have (cf.\ e.g.\  \cite[Section 2.1.2]{rauch}) the energy estimate
\begin{equation}\label{energy}
\|v(t)\|_{L^2}\leq \|v(0)\|_{L^2}+ \int_0^t \|Lv(\sigma)\|_{L^2}\,d\sigma
\end{equation}
for every function $v\in C^1([0,T],\cS(\rd))$ and $t\in [0,T]$. \par
This implies in particular \eqref{stima00} for $N=0$. Assume therefore $N\geq 1$ and let us prove \eqref{stima00} by induction on $N$.
To this end, we apply the above estimate to $v=x^\beta\partial^\alpha u$, $|\alpha|+|\beta|=N$, where $u$ solves $Lu=0$, $u(0)=u_0$. We get
\[
\|x^\beta\partial^\alpha u(t)\|_{L^2}\leq \|x^\beta\partial^\alpha u_0\|_{L^2}+ \int_0^t \|[L,x^\beta\partial^\alpha]u(\sigma)\|_{L^2}\,d\sigma.
\]
We multiply by $\epsilon(t)^{|\alpha|+|\beta|}/(\alpha!\beta!)$, $|\alpha|+|\beta|=N$, and we obtain, since $\epsilon(t)\leq\epsilon(0)=\epsilon_0$,
\begin{multline}\label{astast}
E^{\epsilon(t)}_{N}[u(t)]\leq E^{\epsilon_0}_{N}[u_0]+ \int_0^t \frac{\epsilon(t)^{N}}{\epsilon(\sigma)^{N}}\sum_{|\alpha|+|\beta|=N}\frac{\epsilon(\sigma)^{|\alpha|+|\beta|}}{\alpha!\beta!}\| [L,x^\beta \partial^\alpha]u(\sigma)\|_{L^2}\,d\sigma.
\end{multline}
 Now, we will prove  below the following estimate on the integral involved in the right-hand side of \eqref{astast}. \par
Set, for brevity
\[
\mathcal{E}^{\epsilon_0}_N[u_0]=\sup_{0\leq k\leq N} E^{\epsilon_0}_k[u_0].
\]
\begin{proposition}\label{procomm} There exist constants $C'>0$ and $\overline{\epsilon}_0>0$, depending only on the dimension $d$ and the constant $C_1$ in {\bf (ii)$^\prime$} such that, for every $\epsilon_0\leq \overline{\epsilon}_0$  and $A\geq 1$ in \eqref{epst} we have
\begin{multline}\label{comm}
\int_0^t \frac{\epsilon(t)^{N}}{\epsilon(\sigma)^{N}}\sum_{|\alpha|+|\beta|=N}\frac{\epsilon(\sigma)^{|\alpha|+|\beta|}}{\alpha!\beta!}\|[L,x^\beta\partial^\alpha]u(\sigma)\|_{L^2}\,d\sigma\\
 \leq \int_0^t \frac{\epsilon(t)^{N}}{\epsilon(\sigma)^{N}}C'N E^{\epsilon(\sigma)}_N[u(\sigma)]\,d\sigma+\frac{1}{2}\mathcal{E}^{\epsilon_0}_N[u_0].
\end{multline}
\end{proposition}
We can therefore continue the computation in \eqref{astast} as
\begin{equation}\label{astastast}
E^{\epsilon(t)}_{N}[u(t)]\leq \frac{3}{2}\mathcal{E}^{\epsilon_0}_{N}[u_0]+\epsilon(t)^{N} C'N\int_0^t \frac{E^{\epsilon(\sigma)}_N[u(\sigma)]}{\epsilon(\sigma)^{N}}\,d\sigma.
\end{equation}
We now use the following form of Gronwall inquality, which can be deduced easily from the classical one (see e.g.\ \cite[Lemma 2.1.3]{rauch}).
\begin{lemma}\label{gronwall}
Let $0\leq g,\psi,a\in L^\infty_{loc}([0,T])$, with $a>0$, and $0\leq h\in L^1_{loc}([0,T])$, and
\[
\psi(t)\leq g(t)+a(t)\int_0^t h(\sigma)\frac{\psi(\sigma)}{a(\sigma)}\,d\sigma.
\]
Then, with $H(t):=\int_0^t h(\sigma)\, d\sigma$,
\[
\psi(t)\leq g(t)+a(t) e^{H(t)}\int_0^t e^{-H(\sigma)}h(\sigma)\frac{g(\sigma)}{a(\sigma)}\,d\sigma.
\]
\end{lemma}
We apply Lemma \ref{gronwall} to \eqref{astastast} with $\psi(t)=E^{\epsilon(t)}_{N}[u(t)]$, $g(t)=\frac{3}{2}\mathcal{E}^{\epsilon_0}_{N}[u_0]$, $h(\sigma)=C'N$, $a(t)=\epsilon(t)^{N}$. \par
From \eqref{astastast} we therefore get, if $A>C'$,
\begin{align*}
E^{\epsilon(t)}_{N}[u(t)] &\leq \frac{3}{2}\mathcal{E}^{\epsilon_0}_N[u_0]\Big(1+C'N e^{(C'-A)Nt}\int_0^t e^{(A-C')N\sigma}\,d\sigma\Big)\\
&\leq \frac{3}{2}\mathcal{E}^{\epsilon_0}_N[u_0] \Big(1+\frac{C'}{A-C'}\Big)
\end{align*}
and this last expression is $\leq 2 \mathcal{E}^{\epsilon_0}_N[u_0]$ if $A\geq 4C'$. This concludes the proof of \eqref{stima00}.

\end{proof}
\begin{proof}[Proof of Proposition \ref{procomm}]
It is slightly easier to work with the standard (left) quantization. We have $a^w(t,x,D)=\tilde{a}(t,x,D)$, with $\tilde{a}$ satisfying the same estimates as in {\bf (ii)$^\prime$} for a new constant $\tilde{C}_1$ depending only on the dimension $d$ and the early constant $C_1$ (cf.\ \cite[Chapter XVIII]{hormanderIII} or, more precisely, the proof of \cite[Theorem 1.2.4]{nr}). Hence we can assume that  {\bf (ii)$^\prime$} holds for $\tilde{a}$ as well. \par
Using the inverse Leibniz' formula we can write (cf.\ e.g.\  \cite[Formula (4.5)]{CN2})
\[
([L,x^\beta \partial^\alpha]u)=\sum_{\delta\leq \alpha}\sum_{\gamma\leq\beta\atop (\delta,\gamma)\not=(0,0)}(-1)^{|\gamma|+1}\binom{\beta}{\gamma}\binom{\alpha}{\delta}(D^\gamma_\xi\partial^\delta_x \tilde{a})(t,x,D_x) (x^{\beta-\gamma}\partial^{\alpha-\delta} u).
\]
Now, it follows from the continuity properties of pseudodifferential operators on weighted Sobolev spaces (see e.g.\  \cite[Proposition 1.5.5]{nr}) that if $|\gamma|+|\delta|\geq 1$, in view of the assumption {\bf (ii)$^\prime$} we have
\[
\| (D^\gamma_\xi\partial^\delta_x \tilde{a})(t,x,D_x)u\|_{L^2}\leq C_0 C_1^{|\gamma|+|\delta|}\gamma!\delta!\Big(\|u\|_{L^2}+\sum_{j=1}^d \|x_j u\|_{L^2}+\sum_{j=1}^d\|\partial_{x_j} u\|_{L^2}\Big)
\]
for some constant $C_0$ depending only on the dimension $d$, whereas the constant $C_1$ is the same which appears in {\bf (ii)$^\prime$}.\par
As a consequence we get
\begin{align}\label{es6}
&\frac{\epsilon^{|\alpha|+|\beta|}}{\alpha!\beta!}\|[L,x^\beta\partial^\alpha]u\|_{L^2}\leq C_0
\sum_{\delta\leq \alpha}\sum_{\gamma\leq\beta\atop (\delta,\gamma)\not=(0,0)} (C_1\epsilon)^{|\gamma|+|\delta|}\epsilon^{|\beta-\gamma|+|\alpha-\delta|}\frac{\|x^{\beta-\gamma}\partial^{\alpha-\delta}u\|_{L^2}}{(\beta-\gamma)!(\alpha-\delta)!}\\
&+C_0C_1
\sum_{j=1}^d\sum_{\delta\leq \alpha}\sum_{\gamma\leq\beta\atop (\delta,\gamma)\not=(0,0)}\sum_{j=1}^d (C_1\epsilon)^{|\gamma|+|\delta|-1}\epsilon^{|\beta-\gamma|+|\alpha-\delta|+1}(\beta_j-\gamma_j+1)\frac{\|x^{\beta-\gamma+e_j}\partial^{\alpha-\delta}u\|_{L^2}}{(\beta-\gamma+e_j)!(\alpha-\delta)!}\nonumber\\
&+C_0C_1
\sum_{j=1}^d\sum_{\delta\leq \alpha}\sum_{\gamma\leq\beta\atop (\delta,\gamma)\not=(0,0)} (C_1\epsilon)^{|\gamma|+|\delta|-1}\epsilon^{|\beta-\gamma|+|\alpha-\delta|+1}(a_j-\delta_j+1)\frac{\|x^{\beta-\gamma}\partial^{\alpha-\delta+e_j}u\|_{L^2}}{(\beta-\gamma)!(\alpha-\delta+e_j)!}\nonumber\\
&+ C_0C_1^{-1}
\sum_{j=1}^d\sum_{\delta\leq \alpha}\sum_{\gamma\leq\beta-e_j\atop (\delta,\gamma)\not=(0,0)} (C_1\epsilon)^{|\gamma|+|\delta|+1}\epsilon^{|\beta-\gamma-e_j|+|\alpha-\delta|}\frac{\|x^{\beta-\gamma-e_j}\partial^{\alpha-\delta}u\|_{L^2}}{(\beta-\gamma-e_j)!(\alpha-\delta)!}\nonumber
\end{align}
where $e_j$ denotes the $j$-th element of the canonical basis of $\R^d$ and we used
\[
\partial_{x_j} x^{\beta-\gamma}\partial^{\alpha-\delta}=x^{\beta-\gamma}\partial^{\alpha-\delta+e_j}+(\beta_j-\gamma_j)x^{\beta-\gamma-e_j}\partial^{\alpha-\delta}.
\]
Summing further on $|\alpha|+|\beta|=N$ the first term in the right-hand side of \eqref{es6} is dominated by
\[
C_0 \sum_{|\tilde{\alpha}|+|\tilde{\beta}|\leq N-1}\sum_{|\gamma|+|\delta|=N-|\tilde{\alpha}|-|\tilde{\beta}|} (C_1\epsilon)^{N-|\tilde{\alpha}|-|\tilde{\beta}|}\epsilon^{|\tilde{\alpha}|+|\tilde{\beta}|}
\frac{\| x^{\tilde{\beta}}\partial^{\tilde{\alpha}}u\|_{L^2}}{\tilde{\beta}!\tilde{\alpha}!}.
\]
Since the number of multiindices $(\gamma,\delta)$ satisfying $ |\gamma|+|\delta|=N-|\tilde{\alpha}|-|\tilde{\beta}|$ does not exceed $2^{N-|\tilde{\alpha}|-|\tilde{\beta}|+d-1}$ this last expression is
\[
\leq C'\sum_{k=0}^{N-1} (C'\epsilon)^{N-k} E^{\epsilon}_k[u]
\]
for some constant $C'$ as in the statement. The same holds for the fourth term in the right-hand side of \eqref{es6}. Similary we see that the sum on $|\alpha|+|\beta|=N$ of the second and third term is dominated by
\[
C'\sum_{k=1}^{N} (C'\epsilon)^{N-k} k E^{\epsilon}_k[u].
\]
Summing up we obtain, for a new constant $C'$ as above,
\begin{multline*}
\sum_{|\alpha|+|\beta|=N}\frac{\epsilon(\sigma)^{|\alpha|+|\beta|}}{\alpha!\beta!}\|[L,x^\beta\partial^\alpha]u(\sigma)\|_{L^2}\leq C'N E^{\epsilon(\sigma)}_N[u(\sigma)] \\ +C'\sum_{k=0}^{N-1}(C'\epsilon(\sigma))^{N-k} (k+1) E^{\epsilon(\sigma)}_k[u(\sigma)].
\end{multline*}
Substituting in the left hand side of \eqref{comm} we see that it is sufficient to prove that
\begin{equation}\label{pos2}
\int_0^t \frac{\epsilon(t)^{N}}{\epsilon(\sigma)^{N}} C'\sum_{k=0}^{N-1} (C'\epsilon(\sigma))^{N-k} (k+1) E^{\epsilon(\sigma)}_k[u(\sigma)]\,d\sigma\leq \frac{1}{2}\mathcal{E}^{\epsilon_0}_N[u_0] \  \textrm{for} \  N \geq 1,
\end{equation}
if $\epsilon(0)=\epsilon_0\leq\overline{\epsilon}_0$ with $\overline{\epsilon}_0$ small enough. \par
By the inductive hypothesis we have
\[
E^{\epsilon(\sigma)}_k[u(\sigma)]\leq 2 \mathcal{E}^{\epsilon_0}_N[u_0], \quad k<N,\ \sigma\in[0,T],
\]
and therefore the left-hand side of \eqref{pos2} can be estimated by
\begin{multline*}
2C' \mathcal{E}^{\epsilon_0}_N[u_0]\sum_{k=0}^{N-1} (C'\epsilon_0)^{N-k} (k+1)e^{-ANt}\int_0^t e^{Ak\sigma}\,d\sigma\\
\leq 2C' \mathcal{E}^{\epsilon_0}_N[u_0] \Big((C'\epsilon_0)^{N}te^{-NAt} +\sum_{k=1}^{N-1}  (C'\epsilon_0)^{N-k} \frac{k+1}{Ak}\Big).
\end{multline*}
The last expression in parenthesis is $<(2C')^{-1}/2$ if $\epsilon_0$ is small enough, for every $A\geq 1$, $N\geq 1$, $t\in[0,T]$. \par
This concludes the proof of \eqref{pos2} and therefore \eqref{comm} is proved.
\end{proof}
\section{A special case: metaplectic operators}
Let us consider the special case  when the symbol $a(t,x,\xi)$ in \eqref{equazione} is a second order homogeneous polynomial. We show that the propagator then admits a representation as in \eqref{rap}, with a window $G(t,s,y)$ independent of $x,\xi$. Moreover, it will be evident from the proof that the phase shift $\psi(t,x,\xi)$ in \eqref{rap} comes from the commutator of the propagator $S(t,0)$ and the time-frequency shift $\pi(x,\xi)$. \par
This result applies, in particular, to metaplectic operators, which arise as propagators for the Cauchy problem \eqref{equazione} when $a(t,x,\xi)$ is in addition time-independent (cf.\ \cite[Chapter 4]{folland}). 
\begin{theorem}\label{teo-pre}
Suppose that the symbol $a(t,x,\xi)$ in \eqref{equazione} is a real-valued second order homogeneous polynomial. Let $g\in S^1_1(\rd)$, $\|g\|_{L^2}=(2\pi)^{-d/2}$. Then the propagator $S(t,s)$ can be written in the form
\begin{equation}\label{rap-bis}
S(t,s)f=\int_{\rdd} e^{i\psi(t,x,\xi)-i\psi(s,x,\xi)}\pi(x^t,\xi^t) G(t,s,\cdot)\langle f,\pi(x^s,\xi^s)g\rangle\,dx\,d\xi,
\end{equation}
where $G(t,s,\cdot)=S(t,s)g$ is still in $S^1_1(\rd)$. 
\end{theorem}
\begin{proof}
Without loss of generality we take $s=0$ (hence $x^0=x$, $\xi^0=\xi$, $\psi(0,x,\xi)=0$) and we therefore omit the variable $s$ in the notation for $G$.  
\par
We use the inversion formula\footnote{Our normalization $\|g\|_{L^2}=(2\pi)^{-d/2}$ is different from that in \cite[Corollary 3.2.3]{book} because we chose a different normalization in the definition of the Fourier transform.} \cite[Corollary 3.2.3]{book}
\begin{equation}\label{invformula}
f=\int_{\rdd} \langle f,\pi(x,\xi)g\rangle \pi(x,\xi)g\,dx\,d\xi
\end{equation}
to which we apply the propagator $S(t,0)$. We then obtain the desired formula \eqref{rap-bis} for the window
\[
G(t,x,\xi,y)=e^{-i\psi(t,x,\xi)} \pi(x^t,\xi^t)^\ast S(t,0)[\pi(x,\xi) g](y).
\]
We have to prove that this window is independent of $x,\xi$ and that, in fact,  
\begin{equation}\label{pre-0}
G(t,x,\xi,\cdot)=S(t,0)g.
\end{equation} The fact that it still belongs to $S^1_1(\rd)$ is then an immediate consequence of Theorem \ref{enerinf}.  \par
Now, \eqref{pre-0} is equivalent to
\[
S(t,0) \pi(x,\xi)= e^{i\psi(t,x,\xi)} \pi(x^t,\xi^t) S(t,0).
\]
In other terms, we have to prove that if $v(t,y)$ is a solution of $D_t v+a^w(t,y,D_y) v=0$ then 
\[
u(t,y):=e^{i\psi(t,x,\xi)}\pi(x^t,\xi^t)v(t,y)=e^{i\psi(t,x,\xi)} e^{i\xi^t(y-x^t)}v(t,y-x^t)
\]
is still a solution. To this end, by an explicit computation, using \eqref{sistema} and \eqref{fase}, we have
\begin{align*}
D_t u&=e^{i\psi(t,x,\xi)}\pi(x^t,\xi^t)[D_t v+\partial_t\psi v+\dot{\xi}^t yv-\xi^t \dot{x}^tv-\dot{x}^tD_y v]  \\
&=e^{i\psi(t,x,\xi)}\pi(x^t,\xi^t)[D_t v-a(t,x^t,\xi^t)v-a_x(t,x^t,\xi^t)y v-a_\xi(t,x^t,\xi^t)D_y v]
\end{align*}
whereas (cf.\ the symplectic invariance of the Weyl calculus in \cite[Theorem 18.5.9]{hormanderIII})
\[
a^w(t,y,D_y)u=e^{i\psi(t,x,\xi)}\pi(x^t,\xi^t)a^w(t,y+x^t,D_y+\xi^t)v.
\]
Hence we obtain
\[
D_t u+a^w(t,y,D_y) u=e^{i\psi(t,x,\xi)}\pi(x^t,\xi^t)[D_t v+a_2^w(t,y,D_y) v],
\]
with
\[
a_2(t,y,\eta)= a(t,y+x^t,\eta+\xi^t)-a(t,x^t,\xi^t)-a_x(t,x^t,\xi^t)y-a_{\xi}(t,x^t,\xi^t)\eta=a(t,y,\eta)
\]
where the last equality follows because $a$ is a homogeneous polynomial of order $2$. Hence we get the desired conclusion. 
\end{proof}
\section{Proof of the main result (Theorem \ref{mainteo})}
The proof relies on the energy estimate in the relevant function space $S^1_1(\rd)$, that is Theorem \ref{enerinf}, and some ingenious ideas from \cite[Theorem 5]{tataru} (where the smooth counterpart of the present theorem was proved), although the pattern followed below is technically different.
\par
Let $g$ be as in the statement. As in the proof of Theorem \ref{teo-pre} we take $s=0$ and we can also assume $\|g\|_{L^2}=(2\pi)^{-d/2}$, so that we get the desired formula \eqref{rap} with
\begin{align*}
G(t,x,\xi,y)&=e^{-i\psi(t,x,\xi)} \pi(x^t,\xi^t)^\ast S(t,0)[\pi(x,\xi) g]\\
&=
e^{-i\psi(t,x,\xi)} e^{-i\xi^t y} S(t,0)[\pi(x,\xi) g](y+x^t).
\end{align*}
It remains to prove the estimate
\begin{equation}\label{es1}
|y^\gamma \partial^\alpha_x\partial^\beta_\xi\partial^\nu_y G(t,x,\xi,y)|\leq C_{\alpha,\beta}^{|\gamma|+|\nu|+1}\gamma!\nu!
\end{equation}
for some constants $C_{\alpha,\beta}>0$ independent of $\gamma$ and $\nu$. First we prove this  for $\alpha=\beta=0$. We can argue as in the proof of Theorem \ref{teo-pre} and obtain that $G$ verifies the equation 
\begin{equation}\label{es3}
 \Big(D_t+a_2^w(t,y,D_y)\Big) G=0,\quad G(0)=g
 \end{equation}
 where
 \[
 a_2(t,y,\eta)=a(t,x^t+y,\xi^t+\eta)-a(t,x^t,\xi^t)-y a_x(t,x^t,\xi^t)-\eta a_\xi(t,x^t,\xi^t)
 \]
 Of course, $a_2$ depends also on $x,\xi$ but we omit these parameters in the notation, for brevity. Observe that $a_2$ is still in $S^{(2)}_{a}$ and moreover vanishes of second order at $(y,\eta)=(0,0)$. Therefore by a Taylor expansion at $(0,0)$ we see that it verifies
  \begin{equation}\label{s7}
 |\partial^\gamma_y\partial^\nu_\eta a_2(t,y,\eta)|\leq C^{|\gamma|+|\nu|+1}\gamma!\nu! (1+|y|+|\eta|)^{(2-|\gamma|-|\nu|)_+}
 \end{equation}
 where $(\cdot)_+$ denotes positive part, for some constant $C>0$ independent of $\gamma,\nu$ {\it and the parameters $x,\xi$}. \par
 Hence we are left to prove that the solution $G$ to \eqref{es3} 
is in $ S^1_1(\rd)$ {\it uniformly with respect to $t\in[0,T]$, $x,\xi\in\rd$}. Now, we already know from \cite[Theorem 5]{tataru} that $G\in C^1([0,T],\cS(\rd))$, hence the desired conclusion follows at once from the a priori estimate in Theorem \ref{enerinf} (\eqref{s7} implies the assumption {\bf (ii)$^\prime$} in that theorem, with a constant $C_1$ independent of $x,\xi$; the same uniformity therefore is inherited by the analyticity radius $\epsilon(t)$ in the conclusion). Hence \eqref{es1} is proved for $\alpha=\beta=0$. \par
The proof of \eqref{es1} can now be carried on by induction on $|\alpha|+|\beta|$. We first compute the gradients $\partial_x G$ and $\partial_\xi G$.  \par
We observe that the definition \eqref{fase} implies that\footnote{The left and right-hand side concide for $t=0$ and have the same time derivative: from \eqref{fase} we get $\frac{d}{dt} d\psi(t)=d\frac{d}{dt}\psi(t)=\xi^t da_{\xi}-a_x dx^t$, which equals $\frac{d}{dt}(\xi^tdx^t-\xi dx)$ by \eqref{sistema} (all functions are understood evaluated at $(t,x^t,\xi^t)$).}
$d\psi(t)=\xi^t dx^t-\xi dx$, that is 
\begin{equation}\label{es4}
\partial_x \psi= \xi^t \partial_x x^t-\xi,\qquad \partial_\xi \psi=\xi^t\partial_\xi x^t,
\end{equation}
Using the first equation in \eqref{es4} and the chain rule we get
\begin{align}\label{es5}
\partial_x G&=(-i\partial_x \psi-i\partial_x \xi^t y-i\xi)G+\partial_x x^t e^{-i\psi(t,x,\xi)}\pi(x^t,\xi^t)^\ast \partial_y S(t,0)[\pi(x,\xi)g]\\
&\qquad\qquad\qquad\qquad\qquad\qquad-e^{-i\psi(t,x,\xi)}\pi(x^t,\xi^t)^\ast S(t,0)[\pi(x,\xi)\partial_yg]\nonumber \\
&=\partial_x x^t \partial_y G-i\partial_x \xi^t y G-e^{-i\psi(t,x,\xi)}\pi(x^t,\xi^t)^\ast S(t,0)[\pi(x,\xi)\partial_y g](y).\nonumber
\end{align}
Similarly, using the second equation in \eqref{es4} we get
\begin{align}\label{es6bis}
\partial_{\xi} G&=(-i\partial_\xi \psi-i\partial_\xi \xi^ty)G+\partial_\xi x^te^{-i\psi(t,x,\xi)}\pi(x^t,\xi^t)^\ast\partial_y S(t,0)[\pi(x,\xi)g](y)\\
&\qquad\qquad\qquad\qquad\qquad +ie^{-i\psi(t,x,\xi)}\pi(x^t,\xi^t)^\ast S(t,0)[\pi(x,\xi) y g](y)\nonumber\\
&=\partial_\xi x^t \partial_y G-i\partial_\xi \xi^t y G+ie^{-i\psi(t,x,\xi)}\pi(x^t,\xi^t)^\ast S(t,0)[\pi(x,\xi)y g](y).\nonumber
\end{align}

Now,  it follows from {\bf (ii)} that for the Hamiltonian flow the following estimates hold true, for new constants $C_{\alpha,\beta}$:
\begin{equation}\label{regol}
|\partial^\alpha_x\partial^\beta_\xi x^t|+ |\partial^\alpha_x\partial^\beta_\xi \xi^t|\leq C_{\alpha,\beta},\quad |\alpha|+|\beta|\geq 1,\ t\in[0,T],\ x,\xi\in\rd.
\end{equation}
Hence, we see by induction on $|\alpha|+|\beta|$ and the formulas \eqref{es5} and \eqref{es6bis} that the partial derivatives $\partial^\alpha_x\partial^\beta_\xi G$ are finite sums of terms of the form 
\[
y^\delta \partial^\gamma_y G,\quad |\gamma+\delta|\leq |\alpha+\beta|,
\]
\[
y^\delta \partial^\gamma_y \Big(e^{-i\psi(t,x,\xi)}\pi(x^t,\xi^t)^\ast S(t,0)[\pi(x,\xi)y^\mu \partial^\nu_y g]\Big),\quad |\gamma+\delta+\mu+\nu|\leq |\alpha+\beta|,
\]
possibly multiplied by functions of $t,x,\xi$ having bounded derivatives of every order with respect to $x,\xi$.\par Since $S^1_1(\rd)$ is preserved by derivation and multiplication by polynomials, it follows from the result already proved for $\alpha=\beta=0$ that the estimates \eqref{es1} hold for every $\alpha,\beta$. This concludes the proof. 
\section{Exponential sparsity}
The following characterizations of the space $S^1_1(\rd)$ will be used often in the following and can be found in \cite[Proposition 3.12]{GZ} and \cite[Theorem 6.1.6]{nr}.\par Recall that $V_g f(z)=\langle f,\pi(z) g\rangle$.
\begin{proposition}\label{prop}
Let $g\in S^1_1(\rd)\setminus\{0\}$. Then for $f\in (S^1_1)'(\rd)$ the following conditions are equivalent.\par\medskip
(a) $f\in S^1_1(\rd)$.\par
(b) There exist constants $C,\epsilon>0$ such that
\[
|f(x)|\leq C e^{-\epsilon |x|}\qquad |\widehat{f}(\xi)|\leq C e^{-\epsilon|\xi|},\quad \forall x,\xi\in\rd.
\]
\indent (c) There exist constants $C,\epsilon>0$ such that
\[
|V_g f(z)|\leq C e^{-\epsilon|z|},\quad z\in\rdd.
\]
The above constants $C,\epsilon$ are uniform when $f$ varies in bounded subsets of $ S^1_1(\rd)$. 
\end{proposition}

We have then the following sparsity result. 
\begin{theorem}\label{sparsity}
 Let $g\in S^1_1(\rd)$, and assume   ${\bf (i)-(iii)}$ in the Introduction. There exist constants $C,\epsilon>0$ such that
\begin{equation}\label{sparsityeq}
|\langle S(t,0)\pi(z)g,\pi(w)g\rangle|\leq C\exp\big(-\epsilon|w-\chi_t(z)|\big),\quad w,z\in\rdd,\ 0\leq t\leq T,
\end{equation}
where $\chi_t(x,\xi)=(x^t,\xi^t)$ is the corresponding canonical transformation.
\end{theorem}

\begin{proof}
Using the integral representation in Theorem \ref{mainteo} we get
\begin{multline*}
|\langle S(t,0)\pi(z)g,\pi(w) g\rangle|\\
\leq \int_{\rdd}  |\langle \pi(x^t,\xi^t)G(t,s,x,\xi,\cdot),\pi(w) g\rangle||\langle \pi(z)g,\pi(x,\xi)g\rangle|dx\,d\xi.
\end{multline*}
By Proposition \ref{prop} we can continue the estimate as
\[
\leq C \int_{\rdd} \exp(-\epsilon|(x^t,\xi^t)-w|)\exp(-\epsilon|(x,\xi)-z|)dx\,d\xi.
\]
It follows from the assumption {\bf (ii)} that the canonical transformation $\chi_t$ is Lipschitz together with its inverse\footnote{This can be seen as follows: the functions $(x^t,\xi^t)=\chi_t(x,\xi)$ satisfy the Hamiltonian system for every fixed $x,\xi$. By differentiating the system with respect to $x,\xi$ and applying the chain rule we see that the Jacobian matrix $\chi'_t$ satisfies a linear system with bounded coefficients, uniformly with respect to $x,\xi\in\rd$, $t\in[0,T]$, with initial condition $\chi'_0={\rm Id}$. Hence the entries of that matrix are bounded with respect to $x,\xi\in\rd$, $t\in[0,T]$. The same conclusion holds for $\chi_t^{-1}$, because ${\rm det}\, \chi'_t=1$, being $\chi_t$ symplectic.}, uniformly with respect to $t\in [0,T]$. Since $|(x^t,\xi^t)-w|=|\chi_t(x,\xi)-w|\asymp |(x,\xi)-\chi_t^{-1}(w)|$, it is easy to estimate the above convolution of exponentials (cf.\ \cite[Lemma 4.3]{cnr2}) as
\[
\leq C' \exp(-\epsilon' |z-\chi_t^{-1}(w)|)\leq  C' \exp(-\epsilon'' |w-\chi_t(z)|).
\]
\end{proof}

As a consequence we obtain at once continuity results on the so-called modulation spaces.
We first recall their definition.\par 
We have given in \eqref{stft} the definition of the short-time Fourier transform $V_gf$ of a distribution $f$. We then consider a weight $v$ which is
continuous, positive,  even, submultiplicative function
(submultiplicative weight), i.e., $v(0)=1$, $v(z) =
v(-z)$, and $ v(z_1+z_2)\leq v(z_1)v(z_2)$, for all $z,
z_1,z_2\in\rd.$  Submultiplicativity implies that $v(z)$ is \emph{dominated} by an exponential function, i.e.
\begin{equation} \label{weight}
 \exists\, C, k>0 \quad \mbox{such\, that}\quad  1\leq v(z) \leq C e^{k |z|},\quad z\in \rd.
\end{equation}

For example, every weight of the form
\begin{equation} \label{BDweight}
v(z) =   e^{s|z|^b} (1+|z|)^a \log ^r(e+|z|)
\end{equation}
 for parameters $a,r,s\geq 0$, $0\leq b \leq 1$ satisfies the
above conditions.\par

We denote by $\mathcal{M}_v(\rd)$ the space of $v$-moderate weights on $\rd$;
these  are measurable positive functions $m$ satisfying $m(z+\zeta)\leq C
v(z)m(\zeta)$ for every $z,\zeta\in\rd$. When dealing with possibly exponential weights it is convenient to consider windows in the space $\Sigma^1_1(\rd)$ of functions $f$ satisfying
\[
|x^\alpha\partial^\beta f(x)|\leq C_0C^{|\alpha|+|\beta|+1}\alpha!\beta!, \quad \alpha,\beta\in\bN^d,\ x\in\rd
\]
for some $C_0>0$ and every $C>0$. We also write ${\Sigma^1_1}'(\rd)\supset \cS'(\rd) $ for the corresponding dual space (ultradistributions).
\begin{definition}  \label{prva}
Given  $g\in\Sigma^1_1(\rd)$, a  weight
function $m\in\cM _v(\rdd)$, and $1\leq p,q\leq
\infty$, the {\it
  modulation space} $M^{p,q}_m(\Ren)$ consists of all tempered
ultra-distributions $f\in(\Sigma^1_1)' (\rd) $ such that $V_gf\in L^{p,q}_m(\Renn )$
(weighted mixed-norm spaces). The norm on $M^{p,q}_m(\rd)$ is
\begin{equation}\label{defmod}
\|f\|_{M^{p,q}_m}=\|V_gf\|_{L^{p,q}_m}=\left(\int_{\Ren}
  \left(\int_{\Ren}|V_gf(x,\o)|^pm(x,\o)^p\,
    dx\right)^{q/p}d\o\right)^{1/q}  \,
\end{equation}
(obvious changes if $p=\infty$ or $q=\infty$).
\end{definition}
For $f,g\in \Sigma^1_1(\rd)$ the above integral is convergent and thus $\Sigma^1_1(\rd)\subset M^{p,q}_m(\rd) $, $1\leq p,q \leq\infty$, cf.\ \cite{medit}, with dense inclusion when $p,q<\infty$, cf.\ \cite{elena07}. The spaces $M^{p,q}_m(\rd)$ are Banach spaces and it worth noticing that every nonzero $g\in \Sigma^1_1(\rd)$ gives an equivalent norm in \eqref{defmod}; hence $M^{p,q}_m(\Ren)$ is independent on the choice of $g\in  \Sigma^1_1(\rd)$. 
%We also denote by $\mathcal{M}^{p,q}_m(\Ren)$ the closure of $\Sigma^1_1(\rd)$ in $M^{p,q}_m(\Ren)$ (so that $\mathcal{M}^{p,q}_m(\Ren)=M^{p,q}_m(\Ren)$ if $p,q<\infty$).

 We refer to \cite {fei} and \cite[Chapter 11]{book} for the details and applications of modulation spaces to Time-frequency Analysis and to \cite[Chapter 6]{baoxiang} for a PDEs perspective.\par
We deduce from Theorem \ref{sparsity} the following continuity result.
\begin{corollary}\label{modsp}
Let $m\in \mathcal{M}_v(\rdd)$ and assume the weight $v$ satisfies
 \[
 \int_{\rdd} v(z)\exp(-\epsilon|z|)\, dz<\infty,
 \]
for every $\epsilon>0$.\par
Under the same hypotheses as in Theorem \ref{sparsity}, the propagator $S(t,0)$ defines a bounded operator ${M}_{m\circ\chi_t}^{p,p}(\R^d)\to {M}_{m}^{p,p}(\R^d)$ for every $1\leq p\leq\infty$, $0\leq t\leq T$.
\end{corollary}
\begin{proof}
Let $g \in \Sigma_1^1(\rd)$ with $\| g \|_{L^2}=(2\pi)^{-d/2}$.
By the inversion formula \eqref{invformula}, we have
\[
V_g (S(t,0) f)(u)=\int_{\R^{2d}}\langle S(t,0) \pi(z) g,
\pi(u)g\rangle \, V_g f(z) \, dz.
\]
It therefore suffices to prove that the map $M$
defined by
\begin{equation}\label{e1}
M[G](u)=\int_{\R^{2d}}\langle S(t,0)\pi(z) g,
\pi(u) g \rangle \, G(z) \, dz.
\end{equation}
is continuous from $L^{p,p}_{m\circ\chi_t}(\rdd)$ into $L^{p,p}_m(\rdd)$. Now, by Theorem \ref{sparsity} we have
\begin{equation}\label{weightesp}
|m(u)M[G](u)|\leq C \int_{\rdd} v(u-\chi_t(z))\exp(-\epsilon|u-\chi_t(z))|m(\chi_t(z))G(z)|\, dz.
\end{equation}
Then one performs the change of variable $z=\chi_t^{-1}(\tilde{z})$, whose Jacobian is $=1$, and concludes by the convolution relation $L^1\ast L^{p}\subset L^{p}$ in $\rdd$, because $v(z)\exp(-\epsilon|z|)$ is integrable by hypothesis.
\end{proof}

\begin{corollary}\label{gelfand} Under the  hypotheses of Theorem \ref{sparsity}, the propagator $S(t,0)$ is bounded on $S^1_1(\rd)$, $0\leq t\leq T$. \end{corollary} \begin{proof} This already follows from Theorem \ref{enerinf}. Alternatively, it can be deduced from the exponential sparsity as well, using the same pattern of the previous proof. Namely, let $g \in S_1^1(\rd)\setminus\{0\}$. We have $f\in S^1_1(\rd)$ if and only if the estimate $|V_g f(z)|\lesssim e^{-h|z|}$, $z\in\rdd$, holds for some $h>0$ (Proposition \ref{prop}). Thus, using \eqref{sparsityeq}, we have 
\begin{align*} |V_g (S(t,0) f)(u)|&\lesssim \int_{\rdd} \exp(-\epsilon|u-\chi_t(z)|) \exp(-h|z|)\, dz\\ &\lesssim  \int_{\rdd} \exp(-\epsilon'|\chi_t^{-1}(u)-z|) \exp(-h|z|)\, dz\\ &\lesssim \exp (-\epsilon''|\chi_t^{-1}(u)|)\\ &\lesssim  \exp (-\epsilon'''|u|) \end{align*} where we used that $\chi_t $ is a bi-Lipschitz diffeomorphism. This estimate provides  the claim. \end{proof}
\begin{corollary}\label{gelfand2} Under the hypotheses of Theorem \ref{sparsity}, the propagator $S(t,0)$ is bounded on the dual space $(S^1_1)'(\rd)$, $0\leq t\leq T$. \end{corollary} 
\begin{proof}
This is a consequence of the previous result, by duality. In fact, the propagator $S(t,s)$ is unitary on $L^2(\rd)$, hence $S(t,s)^\ast=S(t,s)^{-1}=\tilde{S}(T-s,T-t)$, $0\leq s\leq t\leq T$, where $\tilde{S}$ is the forward propagator for the equation $-D_t v+a^w(T-t,x,D)v=0$ ($u$ satisfies the equation in \eqref{equazione} if and only if $v(t)=u(T-t)$ is a solution of this one). On the other hand this last equation verifies the same assumptions as that in \eqref{equazione}. 
\end{proof} 
\section{Representation as classical FIO}

Assume now that for a fixed $t\in\bR$ the canonical transformation $(x,\xi)=\chi_t (y,\eta)$ satisfies the additional assumption
\begin{equation}\label{B3}
|{\rm det}\, \partial x/\partial y|>\delta,
\end{equation}
for a suitable $\delta>0$. Then from the mapping $\chi_t$ we can construct, as usual in the theory of FIOs (cf.\ e.g.\  \cite[Theorem
4.3.2.]{Rodino}), a phase function $\Phi=\Phi_t\in \cC^{\infty}(\rdd)$
  uniquely
determined up to a constant, which in the present global framework satisfies the following properties:
\begin{equation}\label{phasedecay}
|\partial_z^\alpha \Phi(z)|\leq C^{|\alpha|+1}\,\alpha !,\quad |\alpha|\geq 2,\quad z\in\rdd;
\end{equation}
\begin{equation}\label{detcond}
 \left|{\rm det}\,
\left(\frac{\partial^2\Phi}{\partial
x_i\partial \eta_l}\Big|_{
(x,\eta)}\right)\right|>\delta>0,\quad
 (x,\eta)\in \R^{2d},\quad\textrm{for some}\ \delta>0.
\end{equation}
The relationships between the map $(x,\xi)=\chi_t(y,\eta)$ is expressed by the equations
\begin{equation}\label{cantra} \left\{
               \begin{array}{l}
               y=\nabla_{\eta}\Phi(x,\eta)
               \\
              \xi=\nabla_{x}\Phi(x,\eta), \rule{0mm}{0.55cm}
               \end{array}
               \right.
\end{equation}
In fact, by
\eqref{B3} and the global inverse function theorem one can construct the map $\chi_t$ from $\Phi$.
Vice-versa, assuming \eqref{phasedecay} and \eqref{detcond} and solving
with respect to $(x,\xi)$ the above system we come back to the canonical transformation $(x,\xi)=\chi_t(y,\eta)$.\par
The following lemma, proved in \cite[Lemma 4.2]{cnr1} clarifies further the relation between the phase $\Phi
$ and the canonical transformation $\chi$ and will be used below.

\begin{lemma}\label{tre1}
With $\Phi$ and $\chi$ as above, we have
\begin{equation}\label{3mezzo}
|\nabla_x\Phi(x',\eta)-\eta'|+|\nabla_\eta\Phi(x',\eta)-x|
\asymp |\chi _1(x,\eta)-x'|+| \chi _2(x,\eta)-\eta'| \,\quad
\forall x,x',\eta, \eta' \in \rd \, .
\end{equation}
\end{lemma}

We will now show that under the assumption \eqref{B3} we can represent the propagator $S(t,0)$ as a Fourier integral operator (FIO) in the so-called type I form
\begin{equation}\label{uno}
T_{\Phi_t,\sigma_t}f(x)=(2\pi)^{-d}\int_{\rd} e^{i \Phi_t(x,\eta)}\sigma_t(x,\eta)\widehat{f}(\eta)\,d\eta,
\end{equation}
for a suitable symbol $\sigma_t$  satisfying
\begin{equation}\label{quattro}
|\partial^\alpha \sigma_t(z)|\lesssim C^{|\alpha|}\alpha!,\  \alpha\in\mathbb{N}^{2d},\ z\in\rdd,\ t\in [0,T].
\end{equation}
To present the result in full generality, we recall the following Kernel Theorem for ultra-distributions \cite{kernelGS}.
 \begin{theorem}\label{kernelT} There exists an isomorphism between the space of linear continuous maps $T$ from $S^1_1(\rd)$ to $(S^1_1)'(\rd)$ and $(S^1_1)'(\rdd)$, which associates to every $T$ a kernel $K_T\in (S^1_1)'(\rdd)$ such that $$\la Tu,v\ra=\la K_T, v\otimes \bar{u}\ra,\quad \forall u,v \in S^1_1(\rd).$$
  \end{theorem}
 We also need the following characterization of the estimates in \eqref{quattro} \cite[Theorem 3.1]{cnr2}.

 \begin{proposition}\label{proad}
Let $g\in S^1_1(\rd)\setminus\{0\}$. For $f\in(S^1_1)'(\rd)$ the following conditions are equivalent:\par\medskip
(a) There exists a constant $C>0$ such that
\begin{equation}\label{smoothf}
|\partial^\a f(x)|\leq C^{|\a|+1}\a!,\quad x\in\rd,\,\a\in\bN^d.
\end{equation}
\noindent
(b)  There exist constants $C,\epsilon>0$ such that
\begin{equation}\label{STFTeps}
|V_g f(x,\xi)|\leq C \exp\big({-\eps|\xi|}\big),\quad x,\xi\in\rd,\,\a\in\bN^d.
\end{equation}
\end{proposition}

We also introduce the Gelfand-Shilov space
\begin{align}\label{sunmezzo}
S^{1/2}_{1/2}(\rd)&=\{f\in\cS(\rd):\\
&\qquad |x^\alpha\partial^\beta f(x)|\leq C^{|\alpha|+|\beta|+1}(\alpha!\beta!)^{1/2}\  \forall\alpha,\beta\in\bN^d,\ \textrm{for some}\ C>0\},\nonumber
\end{align}
 which will be the right space for window functions.\par
 We have the following characterization.\par
 \begin{theorem}\label{caraI}
Fix  $g\in S^{1/2}_{1/2}(\rd)\setminus\{0\}$. Let $T$ be a continuous linear operator $S^1_1(\rd)\to (S^1_1)'(\rd)$ and $\chi:\rdd\to\rdd$  be a smooth symplectic transformation in $\rdd$ satisfying \eqref{B3} (with $(x,\xi)=\chi(y,\eta)$). Moreover assume, for some constant $C>0$,
\begin{equation}\label{chistima}
|\partial_z^\alpha \chi(z)|\leq C^{|\alpha|+1}\alpha!,\quad |\alpha|\geq 1.
\end{equation}
Let $\Phi\in\cC^\infty(\rdd)$ be the corresponding phase function, therefore enjoying \eqref{phasedecay} and \eqref{detcond}.\par Then the following properties are
equivalent. \par\medskip {\it (a)} $T=T_{\Phi,\sigma}$ is a FIO of type 
I for a symbol $\sigma\in\cC^\infty(\rdd)$ satisfying, for some constant $C>0$,
\begin{equation}\label{simbanalitico}
|\partial^\alpha \sigma(z)|\leq C^{|\alpha|+1} \alpha !,\quad z\in\rdd,\ \a\in\bN^{2d}.
\end{equation}
\indent {\it
(b)} There exist constants $C,\epsilon>0$ such that
\begin{equation}\label{sparsityT}
|\langle T\pi(z)g,\pi(w)g\rangle|\leq C\exp\big(-\epsilon|w-\chi(z)|\big),\quad w,z\in\rdd.
\end{equation}
\end{theorem}
\begin{proof} The implication $(a) \Rightarrow (b)$ is proved in \cite[Theorem 3.3]{cnribero}.
Hence, we are left to prove the vice-versa. By the Kernel Theorem for ultra-distributions (Theorem \ref{kernelT}) $T$ can be written as a FIO of type I, with the phase $\Phi$ uniquely (up to a constant) constructed from the canonical transformation $\chi$ (hence enjoying the conditions \eqref{phasedecay} and \eqref{detcond} above) and  for some symbol $\sigma$ in $(S^1_1)'(\rdd)$. 
We have to verify that  the symbol $\sigma$ satisfies \eqref{simbanalitico}.\par
We use some techniques from \cite{CNG}.
To set up notation,  let  $\Phi _{2,z}$ be  the remainder in the second
order Taylor expansion  of the phase
$\Phi $ at $z\in\rdd$, i.e.,
\begin{equation}\label{phi2}
\Phi_{2,z}(w)=2\sum_{|\alpha|=2}\int_0^1
(1-t)\partial^\alpha\Phi(z+tw)dt \frac{w^\alpha}{\alpha!} \, \qquad
z,w\in\rdd \, ,
\end{equation}
and set
\begin{equation}\label{psi}
\Psi_z(w)=e^{ i \Phi_{2,z}(w)}
\overline{g}\otimes\widehat{g}(w).
\end{equation}
We recall the fundamental relation between the Gabor matrix of a FIO
and the STFT (see \eqref{stft} for its definition) of its symbol  from \cite[Proposition 3.2]{cn} and
\cite[Section 6]{cnr1}: for $g\in \cS^{1/2}_{1/2} (\rd )$ we have 
\[
|\langle T \pi(x,\eta)
g,\pi(x',\eta')g\rangle|=|V_{\Psi_{(x',\eta)}}\sigma
((x',\eta),(\eta'-\nabla_x\Phi(x',\eta),x-\nabla_\eta
\Phi(x',\eta)))| \, .
\]
Now, using \eqref{3mezzo},  and writing $u=(x',\eta)$, $v=(\eta',x)$, \eqref{sparsityT}  translates into
\[
|V_{\Psi_u}\sigma(u,v-\nabla\Phi(u))|\leq C'\exp(-\epsilon' |v-\nabla \Phi(u)|),
\]
and then  into the estimate
\begin{equation}\label{due}
\sup_{(u,w)\in\R^{2d}\times \R^{2d}}\exp(\epsilon' |w|) |V_{\Psi_u}\sigma(u,w)|<\infty.
\end{equation}

Now, setting $G=\overline{g}\otimes \widehat{g}\in S^{1/2}_{1/2}(\rdd)$, we shall prove that there exist $C,k>0$ such that 
\begin{equation}\label{eqvg2}
|V_{G^2}\sigma(u,v)|\leq C\exp( -k |v|),\quad u,v\in\rdd.
\end{equation}
This is equivalent to saying that $\sigma$ satisfies \eqref{simbanalitico}  by Proposition \ref{proad}.\par
We can write
\begin{align}\label{tre}
V_{G^2}\sigma(u,v)&=\int e^{- i t v}\sigma(t)\overline{G^2(t-u)}\,dt\nonumber\\
&= \int e^{- i t v}\sigma(t) e^{- i \Phi_{2,u}(t-u)}\overline{G(t-u)} e^{ i \Phi_{2,u}(t-u)}\overline{G(t-u)}\, dt\nonumber\\
&= \int e^{- i t v}\sigma(t) \overline{\Psi_u(t-u)} e^{ i \Phi_{2,u}(t-u)}\overline{G(t-u)}\,dt\nonumber\\
&=\Fur (\sigma T_u\overline{\Psi_u})\ast_v \Fur(T_u(e^{ i
\Phi_{2,u}}\overline{G}))(v)\nonumber\\
&=V_{\Psi_u}\sigma(u,\cdot)\ast \Fur (T_u (e^{ i
\Phi_{2,u}}\overline{G}))(v).
\end{align}
So that 
\begin{equation}\label{tre0}
|V_{G^2}\sigma(u,v)|\lesssim |V_{\Psi_u}\sigma(u,\cdot)|\ast| \Fur (e^{ i
\Phi_{2,u}}\overline{G})(v)|.
\end{equation}
On the other hand, by the Fa\`a di Bruno formula and \eqref{phasedecay} we have the estimates (cf.\ \cite[Lemma 3.1]{cnribero} for detailed computations)
\[
|\partial^{\beta}e^{i \Phi_{2,(v_1,u_2)}(z)}|\leq C^{|\beta|+1} \sum_{j=1}^{|\beta|}\frac{\beta_1!}{j!}\langle z\rangle^{2j},\quad |\beta|\geq1,
\]
which together with the Leibniz' formula and the definition \eqref{sunmezzo} implies that the set $\{e^{ i
 \Phi_{2,u}}\overline{G} : u\in \rdd \}$ is bounded in $S^1_1(\rdd)$. Since the  Fourier transform is continuous on $S^1_1(\rd)$ (cf.\ \cite[Theorem 6.1.6]{nr}), this implies that the set 
 $\{\Fur( e^{ i
  \Phi_{2,u}}\overline{G} : u\in \rdd \})$ is bounded in $S^1_1(\rd)$ too. We can then use the characterization of $S^1_1(\rdd)$ in terms of exponential decay (Proposition 5.1), and we get 
    \begin{equation}\label{tree}
  \sup_{u\in\rdd}|\Fur (e^{ i
  \Phi_{2,u}}\overline{G})(w)|\lesssim \exp(-h|w|)
  \end{equation}
  for some $h>0$.\par
 Taking $k:=\min\{h, \epsilon'\}/2$ and using Young's inequality in \eqref{tre0}, and then \eqref{due} and \eqref{tree} we have 
 \begin{align*}&\sup_{u,v\in\rdd} |V_{G^2}\sigma(u,v)|\exp(k|v|)\\
 &\lesssim 
 \sup_{u,v\in\rdd}\exp(k|v|)|V_{\Psi_u}\sigma(u,v)| \sup_{u}\intrd \exp(k|v|) |\Fur( e^{ i
   \Phi_{2,u}}\overline{G})(v)|\,dv\\
   &\lesssim  \sup_{u,v\in\rdd}\exp(\epsilon'|v|)|V_{\Psi_u}\sigma(u,v)|
   \\
  &\qquad\qquad\qquad\qquad \times \intrd \exp(-h|v|/2) dv \sup_{u,v} \exp(h|v|)|\Fur( e^{ i
       \Phi_{2,u}}\overline{G})(v)|<\infty.
       \end{align*}
Hence \eqref{eqvg2} is satisfied from some $C,k>0$, which concludes the proof. \end{proof}
\section{Propagation of analytic singularities}
We want now to localize in $\rdd$ the analytic singularities of a distribution and study the action on them of $S(t,0)$.\par
For $\Gamma\subset\rdd$ we define the $\delta$-neighborhood $\Gamma_\delta$, $0<\delta<1$, as 
\begin{equation}\label{7.1}
\Gamma_\delta=\{z\in\rdd:\ |z-z_0|<\delta\langle z_0\rangle\ \textrm{for some}\ z_0\in\Gamma\}.
\end{equation}
For future reference, we begin to list some properties of the $\delta$-neighborhoods.
\begin{lemma}\label{lemma7.1}
Given $\delta$, we can find $\delta^\ast$, $0<\delta^\ast<\delta$, such that for every $\Gamma\subset\rdd$
\begin{equation}\label{eq7.2}
\big(\Gamma_{\delta^\ast}\big)_{\delta^\ast}\subset\Gamma_{\delta},
\end{equation}
\begin{equation}\label{eq7.3}
\big(\rdd\setminus\Gamma_{\delta}\big)_{\delta^\ast}\subset \rdd\setminus\Gamma_{\delta^\ast}.
\end{equation}
\end{lemma}
The proof is straightforward. Consider then the map $\chi=\chi_t:\rdd\to\rdd$, defined in the preceding Sections, for a fixed $t$. Observe that $\chi$ and $\chi^{-1}$ are Lipschitz, and hence $\langle \chi(z)\rangle \asymp \langle z\rangle$. 
\begin{lemma}\label{lemma7.2}
For every $\delta$ there exists $\delta^\ast$, $0<\delta^\ast<\delta$, such that for every $\Gamma\subset\rdd$ 
\begin{equation}\label{eq7.4}
\chi(\Gamma_{\delta^\ast})\subset\chi(\Gamma)_\delta,
\end{equation}
\begin{equation}\label{eq7.5}
\chi(\Gamma)_{\delta^\ast}\subset\chi(\Gamma_\delta).
\end{equation}
The constant $\delta^\ast$ depends on $\chi$ and $\delta$ but it is independent of $\Gamma$.
\end{lemma}
\begin{proof}
We first prove \eqref{eq7.4}. Let $w\in\chi(\Gamma_{\delta^\ast})$. Then there exists $z_0\in \Gamma$ such that $w=\chi(z)$ for some $z$ with $|z-z_0|<\delta^\ast\langle z_0\rangle$. \par
On the other hand, taking $w_0=\chi(z_0)\in\chi(\Gamma)$, we have 
\[
|w-w_0|\leq C_1\delta^\ast \langle z_0\rangle\leq C_1C_2\delta^\ast\langle w_0\rangle.
\]
Taking $\delta^\ast$ sufficiently small, the proof of \eqref{eq7.4} is concluded. \par
As for \eqref{eq7.5}, applying $\chi^{-1}$ to both sides and writing $\Lambda=\chi(\Gamma)$, $\Gamma=\chi^{-1}(\Lambda)$, we are reduced to prove
\[
\chi^{-1}(\Lambda_{\delta^\ast})\subset \chi^{-1}(\Lambda)_\delta,
\]
so we come back to \eqref{eq7.4}.
\end{proof}

In the following we shall argue on $f\in (S^1_1)'(\rd)$, and take windows $g\in S^1_1(\rd)$. Then for every $\lambda>0$
\begin{equation}\label{eq7.6}
|V_g f(z)|\lesssim e^{\lambda \langle z\rangle},\quad z\in\rdd.
\end{equation}
This is an obvious variant of \cite[Theorem 2.4]{GZ}.\par
The readers which are more confident with Schwartz distributions may assume $f\in\cS'(\rdd)$ instead, \eqref{eq7.6} being obviously satisfied. \par
\begin{definition}\label{def7.3}
Let $f\in (S^1_1)'(\rd)$, $g\in S^1_1(\rd)\setminus\{0\}$, $\Gamma\subset\rdd$. We say that $f$ is (analytic) regular in $\Gamma$ if there exist $\epsilon>0$ and $\delta>0$ such that 
\begin{equation}\label{eq7.7}
|V_g f(z)|\lesssim e^{-\epsilon\langle z\rangle}\quad {\rm for}\ z\in\Gamma_\delta.
\end{equation}
\end{definition}
Of course, \eqref{eq7.7} gives us some nontrivial information about $f$ only when $\Gamma$ is unbounded. We shall prove later that Definition \ref{def7.3} does not depend on the choice of the window $g\in S^1_1(\rd)$. \par
\begin{theorem}\label{teo7.4}
Let $S(t,0)$ and $\chi_t$ be defined as in the previous Sections, fix $f\in (S^1_1)'(\rd)$ and $\Gamma\subset\rdd$. If $f$ is regular in $\Gamma$, then $S(t,0)f$ is regular in $\chi_t(\Gamma)$.
\end{theorem}
\begin{proof}
For sufficiently small $\delta>0$ and $\epsilon>0$ we have \eqref{eq7.7} in $\Gamma_\delta$ whereas \eqref{eq7.6} is valid in $\rdd$ for every $\lambda>0$. Now, from Theorem \ref{sparsity} we have 
\begin{equation}\label{eq7.8}
V_g(S(t,0) f)(w)=\int k(t,w,z) V_gf(z)\, dz
\end{equation}
with 
\begin{equation}\label{eq7.9}
|k(t,w,z)|=|\langle S(t,0) \pi(z) g,\pi(w)g\rangle|\lesssim e^{-\epsilon'|w-\chi_t(z)|},
\end{equation}
 for some constant $\epsilon'>0$.\par
 We want to show that $S(t,0)f$ is regular in $\chi_t(\Gamma)$. To this end, using \eqref{eq7.5} in Lemma \ref{lemma7.2}, we take first $\delta^\ast<\delta$ such that $\chi_t(\Gamma)_{\delta^\ast}\subset\chi_t(\Gamma_\delta)$ and then using \eqref{eq7.2} in Lemma \ref{lemma7.1} we fix $\delta'<\delta^\ast$ such that 
 \begin{equation}\label{eq7.10}
 \big(\chi_t(\Gamma)_{\delta'}\big)_{\delta'}\subset\chi_t(\Gamma)_{\delta^\ast}\subset \chi_t(\Gamma_\delta).
 \end{equation}
 Note that for $z\not\in\Gamma_{\delta}$, i.e.\ $\chi_t(z)\not\in \chi_t(\Gamma_{\delta})$, and $w\in \chi_t(\Gamma)_{\delta'}$ we have 
 \begin{equation}\label{eq7.11}
 |w-\chi_t(z)|\gtrsim \max\{\langle z\rangle,\langle w\rangle\}
 \end{equation}
 since $\chi_t(z)\not\in\big((\chi_t(\Gamma)_{\delta'}\big)_{\delta'}$ in view of \eqref{eq7.10}, and we may use as well \eqref{eq7.3}.\par 
 We shall prove 
 \begin{equation}\label{eq7.12}
 |V_g(S(t,0)f)(w)|\lesssim e^{-\eta\langle w\rangle}\quad{\rm for}\ w\in\chi_t(\Gamma)_{
 \delta'},
 \end{equation}
 for some $\eta>0$, and with $\delta'$ determined as before. Using \eqref{eq7.8} and \eqref{eq7.9}, we may estimate
 \begin{equation}\label{eq7.13}
 |e^{\eta\langle w\rangle}V_g(S(t,0)f)(w)|\leq \int_{\rdd}I(w,z)\,dz
 \end{equation}
 with 
 \begin{equation}\label{eq7.14}
 I(w,z)= e^{\eta\langle w\rangle} e^{-\epsilon'|w-\chi_t(z)|}|V_g f(z)|.
 \end{equation}
 In view of \eqref{eq7.12}, to prove that $S(t,0)f$ is regular in $\chi_t(\Gamma)$, it will be then sufficient to show that 
 \begin{equation}\label{eq7.15}
 \int_\rdd I(w,z)\, dz\leq C<\infty \quad {\rm for}\ w\in\chi_t(\Gamma)_{\delta'},
 \end{equation}
 for $\eta$ sufficiently small in \eqref{eq7.14}. Let us split the domain of integration in \eqref{eq7.15} into two domains, $\Gamma_\delta$ and $\rdd\setminus\Gamma_\delta$. First for $z\in \rdd\setminus \Gamma_{\delta}$ and $w\in \chi_t(\Gamma)_{\delta'}$ we use \eqref{eq7.11} to estimate for some $\epsilon''>0$
 \begin{equation}\label{eq7.16}
 e^{-\epsilon'|w-\chi_t(z)|}\leq e^{-\epsilon''\langle w\rangle}e^{-\epsilon''\langle z\rangle}.
 \end{equation}
 Hence  for $w\in \chi_t(\Gamma)_{\delta'}$, by using \eqref{eq7.14}, \eqref{eq7.16} and \eqref{eq7.6}
 \[
 \int_{\rdd\setminus\Gamma_{\delta}} I(w,z)\, dz\leq \int_{\rdd} \exp[\eta\langle w\rangle -\epsilon''\langle w\rangle -\epsilon''\langle z\rangle+\lambda\langle z\rangle]\, dz,
 \]
 which is uniformly bounded if we choose $\eta<\epsilon''$ and $\lambda<\epsilon''$.\par On the other hand, by using \eqref{eq7.7} in $\Gamma_{\delta}$ and estimating 
 \[ |w|\leq |w-\chi_t(z))|+|\chi_t(z)| \leq |w-\chi_t(z)|+C\langle z\rangle,
 \]
($|\chi_t(z)|\lesssim\langle z\rangle$ because $\chi_t$ is globally Lipschitz continuous) 
we obtain
 \[
 \int_{\Gamma_{\delta}} I(w,z)\, dz\leq \int_{\rdd} \exp[\eta|w-\chi_t(z)|+\eta C\langle z\rangle -\epsilon'|w-\chi_t(z)|-\epsilon\langle z\rangle]\, dz,
\]
which is uniformly bounded if we choose $\eta<\epsilon'$, $\eta<\epsilon/C$.\par
The proof is complete. 
\end{proof}

Let us now prove that Definition \ref{def7.3} does not depend on the choice of $g\in S^1_1(\rd)$. We need the following lemma, which is an easy variant of \cite[Lemma 11.3.3]{book}. 
\begin{lemma}\label{lemma7.5}
If $f\in (S^1_1)'(\rd)$, $g,h\in S_1^1(\rd)$, $g\not\equiv0$, then 
\[
|V_h f(w)|\leq \frac{1}{\|g\|_{L^2}^2}\big(|V_g f|\ast |V_h g|  \big)(w).
\]
\end{lemma}
\begin{proposition}\label{pro7.6}
Assume that the estimate \eqref{eq7.7} in Definition \ref{def7.3} is satisfied for some $\epsilon>0$, $\delta>0$, and some choice of $g\in S^1_1(\rd)$. Then \eqref{eq7.7} is still satisfied, for some new $\epsilon>0$, $\delta>0$, if we replace $g$ with $h\in S^1_1(\rd)$. 
\end{proposition}
\begin{proof}
Since 
\[
|V_h g(z)|\lesssim e^{-\epsilon'|z|},\quad z\in \rdd,
\]
for some constant $\epsilon'>0$, Lemma 7.5 gives
\[
|V_h f(w)|\lesssim\int e^{-\epsilon'|w-z|}|V_g f(z)|\, dz.
\]
We then split the domain of integration into $\Gamma_\delta$ and $\rdd\setminus\Gamma_\delta$, and argue as in the proof of the preceding Theorem \ref{teo7.4}, being now $\chi_t={\rm identity}$. 
\end{proof}

The following definition allows one to describe the position in phase space of the singularities of a function $f$.
\begin{definition}\label{def7.7}
Given $f\in (S^1_1)'(\rd)$, we shall call filter of the analytic singularities of $f$ the collection of subsets of $\rdd$: 
\[
\Fu(f)=\{\Lambda\subset\rdd:\ f\ \textit{is regular in}\ \Gamma=\rdd\setminus \Lambda\},
\]
cf.\ Definition \ref{def7.3}. 
\end{definition}
$\Fu(f)$ is a filter since if $\Lambda\in \Fu(f)$ and $\Lambda\subset\Lambda'$, then also $\Lambda'\in \Fu(f)$, and moreover if $\Lambda_1,\ldots,\Lambda_n\in \Fu(f)$ then also $\cap_{j=1}^n \Lambda_j\in \Fu(f)$. Note that any neighborhood of $\infty$ i.e.\ the complementary of a bounded set, belongs to $\Fu(f)$. We have $f\in S^1_1(\rd)$ if and only if $\emptyset\in \Fu(f)$, that is equivalent to saying that there exists $\Lambda_1,\ldots,\Lambda_n\in \Fu(f)$ such that $\cap_{j=1}^n \Lambda_j=\emptyset$.\par
In this language, Theorem \ref{teo7.4} can be rephrased as follows. 
\begin{theorem}\label{teo7.8}
For every $f\in (S^1_1)'(\rd)$ and every fixed $t$, $0\leq t\leq T$, we have 
\[\chi_t(\Fu(f))= \Fu(S(t,0)f).
\]
\end{theorem} 

\begin{proof}
The inclusion $\chi_t(\Fu(f))\subset \Fu(S(t,0)f)$ is just a restatement of Theorem \ref{teo7.4}. The opposite inclusion is equivalent to $\chi_t^{-1}\Fu(S(t,0)f)\subset \Fu(f)$, namely to $\chi_t^{-1}\Fu(g)\subset \Fu(S(t,0)^{-1}g)$, which is true by reversing the time (cf.\ the proof of Corollary \ref{gelfand2}). 
\end{proof}

\section*{Acknowledgements}
We would like to thank the anonymous referee for suggesting several improvements.

\end{document}